\documentclass[a4paper,12pt]{article}
\usepackage{cmap}                        % Поддержка поиска русских слов в PDF (pdflatex)
\usepackage[cp1251]{inputenc}            % Выбор языка и кодировки
\usepackage[english]{babel}
\usepackage[left=2cm,right=2cm,top=2cm,bottom=2cm]{geometry} % поля страницы
\usepackage{amsmath, amsthm, amscd, amsfonts, amssymb, graphicx, color}

\newtheorem{thm}{Theorem}[section]
\newtheorem{cor}[thm]{Corollary}
\newtheorem{lem}[thm]{Lemma}
\newtheorem{prop}[thm]{Proposition}
\newtheorem{defn}[thm]{Definition}

\numberwithin{equation}{section}

\begin{document}

\begin{center}\Large
\textbf{Generalized Fitting subgroups of finite groups}
%The influence of generalizations of the Fitting subgroup on the structure of finite groups
\normalsize

\medskip

V.I. Murashka, A.F. Vasil'ev 

\bigskip

\textbf{Abstract} \end{center}

In this paper we consider the Fitting subgroup $F(G)$ of a finite group $G$ and its
 generalizations: the quasinilpotent radical $F^*(G)$ and the generalized Fitting
 subgroup $\tilde{F}(G)$ defined by
 $\tilde{F}(G)\supseteq \Phi(G)$ and $\tilde{F}(G)/\Phi(G)=Soc(G/\Phi(G))$.
 We sum up known properties of $\tilde{F}(G)$ and suggest some new ones. 
  Let $R$ be a subgroup of a group $G$. We shall call a subgroup $H$ of $G$
 the $R$-subnormal subgroup if $H$ is subnormal in $ \langle H,R\rangle$.
 In this work the influence of $R$-subnormal subgroups (maximal, Sylow, cyclic primary) on
 the structure of finite groups are studied in the case when $R\in\{F(G), F^*(G),\tilde{F}(G)\}$.

\textbf{Keywords:} Finite group, the Fitting subgroup, the quasinilpotent radical,
 the generalized Fitting subgroup, subnormal subgroup, nilpotent group, 
supersoluble group

\textbf{Mathematic Subject Classification (2010)}: 20D15, 20D20, 20D25, 20D35, 20E28.

\section{\bf Introduction}
\vskip 0.4 true cm

All the considered groups are finite. In 1938 H. Fitting \cite{Fit} showed that a product of
 two normal nilpotent subgroups is again nilpotent subgroup. It means that in every group
 there is the unique maximal normal nilpotent subgroup $F(G)$ called the Fitting subgroup.
 This  subgroup has a great influence on the structure of a solvable group. For example
 Ramadan \cite{35} proved
the following theorem:

\begin{thm} Let $G$ be a soluble group. If all maximal subgroups of Sylow subgroups of
 $F(G)$ are normal in $G$ then $G$ is supersoluble.
\end{thm}

  Analyzing proofs of such kind's theorems in solvable case one can note that the
 following properties of the Fitting subgroup $F(G)$  are often used:

 (1) $C_G(F(G))\subseteq F(G)$;

 (2) $\Phi(G)\subseteq F(G)$ and $F(G/\Phi(G))=F(G)/\Phi(G)$;

 (3) $F(G)/\Phi(G)\leq Soc(G/\Phi(G))$.

  But  for the Fitting subgroup of an arbitrary group holds only (2) and (3). Note that there are
 many groups $G$ with $F(G)=1$.   That is why there were attempts to generalize the Fitting subgroup.

   In 1970 H. Bender \cite{18} introduced  the quasinilpotent radical $F^*(G)$. It can be defined by
 the formula $F^*(G)/F(G)=Soc(C_G(F(G))F(G)/F(G))$ and   can be viewed  as a generalization of the
 Fitting subgroup. For $F^*(G)$ the statements like (1) and (3)  holds. This subgroup proved useful
 in the classification of finite simple groups. Also $F^*(G)$ was used by many authors  in the study
 of nonsimple  groups (for example see \cite{37},\cite{38} and \cite{45}, etc.).

   In 1985  P. Forster \cite{F2} showed that  there is the unique characteristic subgroup
 $\tilde{F}(G)$ ($F'(G)$ in Forster notation) in every  group $G$  which satisfies the statements
 like (1)-(3). This subgroup can be defined by

 (1)  $\Phi(G)\subseteq \tilde{F}(G)$;

 (2) $\tilde{F}(G)/\Phi(G)=Soc(G/\Phi(G))$.

Firstly subgroup with this properties was mentioned by P. Shmid \cite{17} in 1972.  It was defined
 in explicit form  by L. Shemetkov in 1978 (see \cite{13}, p.79). P. Shmid and L. Shemetkov used this
 subgroup in the study of stable groups of automorphisms for groups.
  In \cite{33} A. Vasil'ev and etc. proved the following theorem.

  \begin{thm} The intersection of all maximal subgroups $M$ of a group $G$ such that $M\tilde{F}(G)=G$
 is the Frattini subgroup $\Phi(G)$
 of $G$.\end{thm}

  Another direction of applications of generalizations of the Fitting subgroup is  connected with the
 following concept.  Recall \cite{BB},\cite{41} that a subgroup functor is a function $\tau$ which
 assigns to each group $G$ a possibly empty set $\tau(G)$ of subgroups of $G$ satisfying
 $f(\tau(G))=\tau(f(G))$ for any isomorphism $f: G\rightarrow G^*$.

\begin{defn} Let $\theta$ be a subgroup functor  and $R$ be a subgroup of a group $G$. We shall
 call a subgroup $H$ of $G$ the $R$-$\theta$-subgroup if $H\in\theta(\langle H, R\rangle)$.\end{defn}

Let $\theta$ be the $\mathbb{P}$-subnormal subgroup functor. Recall \cite{32} that a subgroup $H$
 of a group $G$ is called  $\mathbb{P}$-subnormal in $G$ if $H=G$ or there is a chain of subgroup
 $H=H_0<H_1<\dots<H_n=G$ where $|H_i:H_{i-1}|$ is a prime for $i=1,\dots,n$.     O. Kramer's theorem
 (\cite{22}, p.12) states

\begin{thm}
If every maximal subgroup of a soluble group $G$ is $F(G)$-$\mathbb{P}$-subnormal then $G$ is
 supersoluble.\end{thm}

This theorem was generalized by Yangming Li,  Xianhua Li in \cite{30}.

\begin{thm}
A group $G$ is supersoluble if and only if every maximal subgroup of  $G$ is
 $\tilde{F}(G)$-$\mathbb{P}$-subnormal.\end{thm}

Let $\theta$ be the conjugate-permutable subgroup functor. Recall \cite{7} that a subgroup $H$ of
 a group $G$  is called conjugate-permutable if $HH^x=H^xH$ for all $x\in G$.

\begin{thm} \cite{31} A group $G$ is nilpotent if and only if every maximal subgroup of
  $G$ is $\tilde{F}(G)$-conjugate-permutable.\end{thm}

\begin{thm} \cite{31} A group $G$ is nilpotent if and only if every Sylow subgroup of
  $G$ is $F^*(G)$-conjugate-permutable.\end{thm}

Let $\theta$ be the subnormal subgroup functor. In  \cite{mv} V. Murashka and A. Vasil'ev began
 to study products of $F(G)$-subnormal subgroups. In this paper we continue to study the
 influence of $R$-subnormal subgroups (maximal, Sylow, cyclic primary) on the structure of
 finite groups  in the case when $R\in\{F(G), F^*(G),\tilde{F}(G)\}$.

\section{Preliminaries}

We use standard notation and terminology, which if necessary can be found in \cite{BB}, \cite{19}
 and \cite{13}. Recall that for a group $G$,  $\Phi(G)$ is the Frattini subgroup of $G$; $\Delta(G)$
 is the intersection of all abnormal  maximal subgroups of $G$; $Z(G)$ is the center of $G$;
 $Z_\infty(G)$ is the hypercenter of $G$; $Soc(G)$ is the socle of $G$; $G_\mathfrak{F}$ is the
 $\mathfrak{F}$-radical of $G$ for a $N_0$-closed class $\mathfrak{F}$ with 1; $G^\mathfrak{F}$
 is the $\mathfrak{F}$ residual of $G$ for a formation $\mathfrak{F}$; $\mathfrak{N}$ is the class
 of all nilpotent groups, $\mathfrak{N}^*$ is the class of all quasinilpotent groups; $\mathbb{F}_p$
 is a field composed by $p$ elements.

A class of group $\mathfrak{F}$ is said to be $N_0$-closed if $A,B\triangleleft G$ and
 $A,B\in\mathfrak{F}$ imply $AB\in\mathfrak{F}$.

A class of group $\mathfrak{F}$ is said to be $s_n$-closed if $A\triangleleft G$ and
 $G\in\mathfrak{F}$ imply  $A\in\mathfrak{F}$.

%We freely use the properties of subnormal subgroups that can be found in chapter 6 of \cite{BB}.

Recall that a subgroup $H$ of a group $G$ is called abnormal in $G$ if $g\in\langle H,H^g\rangle$
 for all $g\in G$. It is known that normalizer of Sylow subgroups and maximal subgroups of $G$
 are abnormal.

\begin{lem} (\cite{13}, p.95 or \cite{40})
Let $G$ be a group. Then \[\Delta(G/\Phi(G))=\Delta(G)/\Phi(G)=Z_\infty(G/\Phi(G))=Z(G/\Phi(G))\]
 in particular $\Delta(G)$ is nilpotent.
\end{lem}

\begin{lem} (\cite{HD}, p.390)
Let $G$ be a group. Then $G^\mathfrak{N}\leq C_G(Z_\infty(G))$.
\end{lem}

The following lemma is obvious.

\begin{lem}
Let $G$ be a group. If  $G^\mathfrak{N}\leq Z_\infty(G)$ then $G$ is nilpotent.
\end{lem}

\begin{lem}(\cite{19}, p. 127) Let $G$ be a group. Then  $C_G(F^*(G))\leq F(G)$.\end{lem}

The following lemma is a simple generalization of Baer's theorem (\cite{22}, p.8).

\begin{lem} Let a group $G=AB$ be a product of supersoluble subnormal subgroups $A$ and $B$.
 If $G'$ is nilpotent then $G$ is supersoluble.\end{lem}

Also we need the following well known result.

\begin{lem} Let a group $G=AB$ be a product of supersoluble subnormal subgroup $A$ and
 nilpotent normal subgroup $B$. Then $G$ is supersoluble.\end{lem}

\begin{lem}Let a group $G=\langle A,B\rangle$ where $A$ and $B$ are subgroups of $G$. Then

(1) $[A,B]\triangleleft G$;

(2) $A[A,B]\triangleleft G$ and  $B[A,B]\triangleleft G$;

(3) $G=AB[A,B]$;

(4) $G'=A'B'[A,B]$. \end{lem}

\begin{proof} (1), (2) and (3) follows from (a),(b) and (h) of lemma 7.4  \cite{HD} p.23.
 Let prove (4). It is clear $A'B'[A,B]\subseteq G'$. Hence by (3) it is sufficient to
 show that $[a_1b_1c_1,a_2b_2c_2]\in A'B'[A,B]$ for every $a_1,a_2\in A$, $b_1,b_2\in B$
 and $c_1,c_2\in [A,B]$.  With the help of lemma 7.3 of \cite{HD} p.22 and (1) it is
 straightforward to check that $A'B'[A,B]$ is normal subgroup of $G$ and we can write
 $[a_1b_1c_1,a_2b_2c_2]$ as the product of conjugates in $G$ of
 $[a_1,a_2], [a_1,b_2], [a_1,c_2], [b_1,b_2], [b_1,c_2], [c_1,c_2]$.
 Hence $[a_1b_1c_1,a_2b_2c_2]\in A'B'[A,B]$. It means that $G'=A'B'[A,B]$.    \end{proof}

\section{Generalizations of $F(G)$ and their properties}

It is well known that $F(F(G))=F(G)$ and $F^*(F^*(G))=F^*(G)$.  In \cite{F1} P. Forster
  showed that there is a group $G$  with
$\tilde{F}(\tilde{F}(G))<\tilde{F}(G)$. He shows that there is a
  a non-abelian simple
group $E$ which has   $\mathbb{F}_pE$-module $V$ such
 that $R=Rad(V)$ is faithful irreducible $\mathbb{F}_pE$-module and $V/R$ is
 irreducible trivial $\mathbb{F}_pE$-module. Let $H$ be the semidirect product
 $V\leftthreetimes E$. Then $H'=RE$ is a primitive group  and $|H:H'|=p$.
  There is $\mathbb{F}_qH$-module $W$   with $C_H(W)=H'$, where $q\not= p$.
 Let $G=W\leftthreetimes H$.
 Then $\Phi(G)=\Phi(H)=R$ and $Soc(G/R)=W\times ER/R$. So $\tilde{F}(G)=W\times ER$
 and $\Phi(\tilde{F}(G))=1$. It means that
 $\tilde{F}(\tilde{F}(G))=Soc(\tilde{F}(G))=R\times W<\tilde{F}(G)$.
 This example  led us to the following definition.

\begin{defn}
Let $G$ be a finite group.   For any nonnegative integer $n$ define the subgroup $\tilde{F}^n(G)$
 by: $\tilde{F}^0(G)= G$ and $\tilde{F}^n(G)=\tilde{F}(\tilde{F}^{n-1}(G))$ for $n>0$.
\end{defn}

It is clear that $\tilde{F}^{i}(G)=\tilde{F}^{i-1}(G)=\tilde{F}(\tilde{F}^{i-1}(G))$ for some $i$.
 So we can define the subgroup $\tilde{F}^\infty(G)$ as the minimal subgroup in the series
 $G=\tilde{F}^0(G)\supseteq \tilde{F}^1(G)\supseteq\dots$. It is clear that
 $\tilde{F}^\infty(G)=\tilde{F}^\infty(\tilde{F}^\infty(G))$.

\begin{prop} Let $n$ be a natural number, $N$ and $H$  be  normal subgroups of a group $G$. Then

(1)  If $N\leq\Phi(\tilde{F}^{n-1}(G))$ then $\tilde{F}^n(G/N)=\tilde{F}^n(G)/N$;

(2)  $F^*(G)\subseteq \tilde{F}^n(G)$;

(3) If $\Phi(\tilde{F}^{n-1}(G))=1$ then $\tilde{F}^n(G)=F^*(G)$;

(4)  $C_{G}(\tilde{F}^n(G))\subseteq F(G)$;

(5)  $\tilde{F}^n(N)\leq \tilde{F}^n(G)$;

(6)  $\tilde{F}^n(G)N/N\leq \tilde{F}^n(G/N)$;

(7) If $G=N\times H$ then $\tilde{F}^n(G)=\tilde{F}^n(H)\times\tilde{F}^n(N)$.
\end{prop}

\begin{proof} (1) When $n=1$ it is directly follows from the definition of $\tilde{F}(G)$
 and $\Phi(G/N)=\Phi(G)/N$. By induction by $n$ we obtain this statement.

(2) The proof was proposed by L. Shemetkov  to the authors in case of $n=1$. Let a group $G$ be
 the minimal order counterexample for
(2).
If $\Phi(G)\neq E$ then for $G/\Phi(G)$ the statement is true. From
 $F^{*}(G)/\Phi(G)\subseteq F^{*}(G/\Phi(G))$ and  $\tilde{F}(G/\Phi(G))=\tilde{F}(G) /\Phi(G)$
we have that $F^{*}(G)\subseteq \tilde {F}(G)$. It is a contradiction with the choice of $G$.

Let $\Phi(G) = 1$. Now $\tilde{F}(G)= Soc(G)$. By 13.14.X  \cite{19}  $F^{*}(G) = E(G)F(G)$.
 Note $\Phi(E(G)) = 1$. Since 13.7.X  \cite{19}  $E(G)/Z(E(G))$ is the direct product of simple
 nonabelian groups, $Z(E(G)) = F(E(G))$. From it and theorem 10.6.A \cite{HD} we conclude that
 $E(G) = HZ(E(G))$ where $H$ is the complement to $Z(E(G))$ in $E(G)$. Now $H$ is the direct product
 of simple nonabelian groups. Since $H char E(G)\triangleleft G$, we have $H\triangleleft G$.
 Note  $H\subseteq Soc(G)$. Since  $Z(E(G))\subseteq F(G)\subseteq \tilde{F}(G)$
   and $H\subseteq Soc(G)$, it follows that
  $E(G)\subseteq \tilde{F}(G)$. Now $F^{*}(G) = E(G)F(G)\subseteq \tilde{F}(G)$.
 It is a contradiction with the choice of $G$.

Assume that $F^*(G)\subseteq \tilde{F}^n(G)$ for $n\geq 1$. It means that  $F^*(\tilde{F}^n(G))=F^*(G)$. By induction $F^*(G)\subseteq \tilde{F}^{n+1}(G)$.

(3) If $\Phi(\tilde{F}^{n-1}(G))=1$ then $\tilde{F}^n(G)$ is the socle of $\tilde{F}^{n-1}(G)$
 and hence $\tilde{F}^n(G)$ is quasinilpotent. From $F^*(G)\subseteq \tilde{F}^n(G)$ it follows
 that $\tilde{F}^n(G)=F^*(G)$.

(4) From $F^*(G)\leq \tilde{F}^n(G)$ it follows that $C_G(\tilde{F}^n(G))\leq C_G(F^*(G))$.
 Since $C_G(F^*(G))\leq F(G)$ by lemma 2.4, we see that $C_G(\tilde{F}^n(G))\leq F(G)$.

(5) Since $\Phi(N)\leq\Phi(G)$, we see that $\tilde{F}(G/\Phi(N))=\tilde{F}(G)/\Phi(N)$.
 Note that $\tilde{F}(N)/\Phi(N)$ is quasinilpotent. Hence $\tilde{F}(N)/\Phi(N)\subseteq F^*(G/\Phi(N))\subseteq\tilde{F}(G/\Phi(N))=\tilde{F}(G)/\Phi(N)$. Thus $\tilde{F}(N)\leq\tilde{F}(G)$. By induction $\tilde{F}^n(N)\leq\tilde{F}^n(G)$.

(6) Note that \[\tilde{F}(G)N/N/\Phi(G)N/N\simeq \tilde{F}(G)N/\Phi(G)N\simeq\tilde{F}(G)/\tilde{F}(G)\cap\Phi(G)N\]
 From $\Phi(G)\subseteq \tilde{F}(G)\cap\Phi(G)N$ it follows that $\tilde{F}(G)N/N/\Phi(G)N/N$
 is quasinilpotent. Since $\Phi(G)N/N\subseteq\Phi(G/N)$, we see that $\tilde{F}(G)N/N\leq\tilde{F}(G/N)$.

 Assume that      $\tilde{F}^n(G)N/N\leq\tilde{F}^n(G/N)$ for some $n\geq1$. Now
 \[\tilde{F}^{n+1}(G/N)=\tilde{F}(\tilde{F}^n(G/N))\geq\tilde{F}(\tilde{F}^n(G)N/N)\geq\]\[\geq\tilde{F}(\tilde{F}^n(G)N)N/N\geq \tilde{F}(\tilde{F}^n(G))N/N=\tilde{F}^{n+1}(G)N/N\] by the previous step and (5).

(7) Assume that the statement is false for $n=1$. Let a group $G$ be a counterexample of
 minimal order. Assume that $\Phi(G)\neq 1$.  Then
\[G/\Phi(G)= H\Phi(G)/\Phi(G)\times N\Phi(G)/\Phi(G)\] Note
 that \[ H\Phi(G)/\Phi(G) \simeq H/H\cap\Phi(G) = H/H\cap(\Phi(H)\times\Phi(N))=H/\Phi(H)\]
By analogy $N\Phi(G)/\Phi(G)\simeq N/\Phi(N)$. So
$\tilde{F}(G/\Phi(G))\simeq \tilde{F}(H/\Phi(H))\times\tilde{F}(N/\Phi(N))$.

From \[\tilde{F}(G/\Phi(G))=\tilde{F}(G)/\Phi(G), \, \tilde{F}(N/\Phi(N))=\tilde{F}(N)/\Phi(N) \, and \, \tilde{F}(H/\Phi(H))=\tilde{F}(H)/\Phi(H)\] it follows that
$\tilde{F}(G)/\Phi(G)\simeq \tilde{F}(H)/\Phi(H)\times\tilde{F}(N)/\Phi(N)$. Now
\[\frac{|\tilde{F}(G)|}{|\Phi(G)|}= \frac{|\tilde{F}(H)|}{|\Phi(H)|}\cdot\frac{|\tilde{F}(N)|}{|\Phi(N)|}\]

 From $\Phi(G)=\Phi(N)\times\Phi(H)$ it follows that $|\tilde{F}(G)|=|\tilde{F}(N)||\tilde{F}(H)|$.
 From (5) it follows that $\tilde{F}(N)\leq \tilde{F}(G)$ and $\tilde{F}(H)\leq \tilde{F}(G)$.
 Thus $\tilde{F}(G)=\tilde{F}(H)\times\tilde{F}(N)$, a contradiction.

  Now   $\Phi(G)= 1$. So $\tilde{F}(G)=F^*(G)$, $\tilde{F}(N)=F^*(N)$ and $\tilde{F}(H)=F^*(H)$.
 It is well known that $F^*(H\times N)=F^*(H)\times F^*(N)$, the final contradiction.

  By induction we have that
$\tilde{F}^n(G)=\tilde{F}^n(H)\times\tilde{F}^n(N)$.
  \end{proof}

From proposition 3.2 follow  properties of $\tilde{F}(G)$.

\begin{cor} Let $N$ and $H$  be  normal subgroups of a group $G$. Then

(1) \cite{30} If $N\leq\Phi(G)$ then $\tilde{F}(G/N)=\tilde{F}(G)/N$;

(2) \cite{33},\cite{30} $F^*(G)\subseteq \tilde{F}(G)$;

(3) \cite{F1} If $\Phi(G)=1$ then $\tilde{F}(G)=F^*(G)$;

(4) \cite{17} $C_{G}(\tilde{F}(G))\subseteq \tilde{F}(G)$;

(5) \cite{F2} $\tilde{F}(N)\leq \tilde{F}(G)$;

(6) \cite{F2} $\tilde{F}(G)N/N\leq \tilde{F}(G/N)$;

(7) If $G=N\times H$ then $\tilde{F}(G)=\tilde{F}(H)\times\tilde{F}(N)$.
\end{cor}

Also we obtain new properties of $\tilde{F}^\infty(G)$.

\begin{cor} Let $N$  and $H$  be  normal subgroups of a group $G$. Then

(1) $\tilde{F}^\infty(G)/\Phi(\tilde{F}^\infty(G))$ is quasinilpotent.

(2)  $F^*(G)\subseteq \tilde{F}^\infty(G)\subseteq\tilde{F}(G)$;

(3) If $\Phi(\tilde{F}^\infty(G))=1$ then $\tilde{F}^\infty(G)=F^*(G)$;

(4)  $C_{G}(\tilde{F}^\infty(G))\subseteq F(G)$;

(5)  $\tilde{F}^\infty(N)\leq \tilde{F}^\infty(G)$;

(6)  $\tilde{F}^\infty(G)N/N\leq \tilde{F}^\infty(G/N)$;

(7)  If $G=N\times H$ then $\tilde{F}^\infty(G)=\tilde{F}^\infty(H)\times\tilde{F}^\infty(N)$.
\end{cor}

Let consider another direction of generalization of the Fitting subgroup. A subgroup functor
 $\tau$ is called $m$-functor if $\tau(G)$ contains $G$ and some maximal subgroups of $G$ for
 every group $G$. Recall (\cite{41}, p.198) that $\Phi_\tau(G)$ is the intersection of all
 subgroups from $\tau(G)$.

\begin{defn}
Let $\tau$ be $m$-functor.  For every group $G$ subgroup $\tilde{F}_\tau(G)$ is defined by

1) $\Phi_\tau(G)\subseteq\tilde{F}_\tau(G)$;

2) $\tilde{F}_\tau(G)/\Phi_\tau(G)=Soc(G/\Phi_\tau(G))$.
\end{defn}

 If $\tau(G)$ is the set of all maximal  subgroups of $G$ for any group $G$ then we obtain
 the definition of $\tilde{F}(G)$.  If $\tau(G)$ is the set of all maximal abnormal subgroups
 for any group $G$ then  $\Phi_\tau(G)=\Delta(G)$. Subgroup $\tilde{F}_\tau(G)=\tilde{F}_\Delta(G)$
 was introduced by  M. Selkin and R. Borodich \cite{36}.

\begin{prop} Let  $G$ be a group. Then $\Delta(G)\subseteq\tilde{F}(G)$ and
 $\tilde{F}(G/\Delta(G))=\tilde{F}(G)/\Delta(G)$.\end{prop}

\begin{proof} From lemma 2.1 it follows that $\Delta(G)\subseteq\tilde{F}(G)$. Let a group $G$ be
 a counterexample of minimal order to the second statement of proposition. Assume that $\Phi(G)\neq 1$.
  By inductive hypothesis
 \[\tilde{F}((G/\Phi(G))/\Delta(G/\Phi(G)))=\tilde{F}(G/\Phi(G))/\Delta(G/\Phi(G))\]
 Now $\Delta(G/\Phi(G))=\Delta(G)/\Phi(G)$ by
 lemma 2.1
 and $\tilde{F}(G/\Phi(G))=\tilde{F}(G)/\Phi(G)$ by proposition 3.2 (1).
\[\tilde{F}((G/\Phi(G))/\Delta(G/\Phi(G)))=
\tilde{F}(G/\Phi(G)/\Delta(G)/\Phi(G))\simeq\tilde{F}(G/\Delta(G))\]
   \[\tilde{F}(G/\Phi(G))/\Delta(G/\Phi(G))=\tilde{F}(G)/\Phi(G)/\Delta(G)/
\Phi(G)\simeq\tilde{F}(G)/\Delta(G)\]

Hence $|\tilde{F}(G)/\Delta(G)|=|\tilde{F}(G/\Delta(G))|$. By 3.2 (6) it follows that
 $\tilde{F}(G)/\Delta(G)\leq\tilde{F}(G/\Delta(G))$. Thus
 $\tilde{F}(G)/\Delta(G)=\tilde{F}(G/\Delta(G))$, a contradiction.

Now $\Phi(G)=1$. It means that $\Delta(G)=Z(G)\leq Soc(G)$. Now $Z(G)$ is
 complemented in $G$. Let $M$ be a complement of $Z(G)$ in $G$. Hence $M\triangleleft G$
 and $G=M\times Z(G)=M\times \Delta(G)$.  Note that $M\simeq G/\Delta(G)$. From proposition 3.2 (7)
 it follows that $\tilde{F}(G)=\tilde{F}(M)\times \tilde{F}(\Delta(G))=\tilde{F}(M)\times \Delta(G)$.
 Thus $\tilde{F}(G/\Delta(G))\simeq\tilde{F}(M)\simeq\tilde{F}(G)/\Delta(G)$. From proposition
 3.2 (6) we obtain the final contradiction. \end{proof}

\begin{cor}
Let $G$ be a group. Then $\tilde{F}_\Delta(G)=\tilde{F}(G)$.
\end{cor}

\begin{prop} Let $G$ be a group. Then $\tilde{F}_\tau(G)\geq\tilde{F}(G)$. In particular,
 $C_G(\tilde{F}_\tau(G))\leq\tilde{F}_\tau(G)$.\end{prop}

\begin{proof} From proposition 3.2 it follows that if $N\triangleleft G$ and $\Phi(G)\subseteq N$
 then $\tilde{F}(G)N/N$ is quasinilpotent. Also note that $\Phi(G/\Phi_\tau(G))=1$.
 By proposition 3.2 $F^*(G/\Phi_\tau(G))=Soc(G/\Phi_\tau(G))$.
 Hence $\tilde{F}(G)\Phi_\tau(G)/\Phi_\tau(G)\leq Soc(G/\Phi_\tau(G))$.
 Thus $\tilde{F}_\tau(G)\geq\tilde{F}(G)$. The second statement follows from
 (4) of proposition 3.2.\end{proof}

\textbf{Problem 1.} Is there a natural number $n$ such that $\tilde{F}^n(G)=\tilde{F}^\infty(G)$
 for any group $G$?

\section{Influence of generalized Fitting subgroups on the structure of groups}

It is well known that a group $G$ is nilpotent if and only if every maximal subgroup of $G$
 is normal in $G$.

 \begin{thm} A group $G$ is nilpotent if and only if every maximal subgroup of $G$ is
 $\tilde{F}(G)$-subnormal.\end{thm}

 \begin{proof}  Let $G$ be a nilpotent group. It is clear that every maximal subgroup of $G$
 is $\tilde{F}(G)$-subnormal.

  Conversely. Assume the theorem is false and let $G$ have minimal order among the
 nonnilpotent groups whose maximal subgroups are  $\tilde{F}(G)$-subnormal.

 Let $\Phi(G)\neq 1$.  Since $ \tilde {F} (G / \Phi (G)) = \tilde {F} (G) / \Phi (G) $, we
 see that every maximal subgroup of $ G / \Phi (G) $ is $ \tilde {F} (G / \Phi (G)) $-subnormal.
 Since $ |G|>|G/\Phi(G)|$, we have $G/\Phi(G)$ is nilpotent. From theorem 9.3(b) (\cite{HD}, p.30)
 it follows that $G$ is nilpotent, a contradiction.

Assume that $ \Phi (G) = 1 $. Then   $ \tilde {F} (G) = Soc (G) $. Let $ M $ be a abnormal maximal
  subgroup of $G$. Since $M$ is $\tilde{F}(G)$-subnormal maximal subgroup, we see that
 $M\triangleleft M\tilde{F}(G)$. Hence $M\supseteq\tilde{F}(G)$. It means that
 $\tilde{F}(G)\leq\Delta(G)$. From $\Phi(G)=1$ and lemma 2.1 it follows that
 $\tilde{F}(G)\leq Z_\infty(G)$. Since $C_G(\tilde{F}(G))\leq\tilde{F}(G)$ by
 proposition 3.1 and $G^\mathfrak{N}\leq C_G(Z_\infty(G))$ by lemma 2.2, we see that
 $G^\mathfrak{N}\leq\tilde{F}(G)\leq Z_\infty(G)$. By lemma 2.3 $G$ is nilpotent,
 the  contradiction. \end{proof}

\begin{cor}\cite{33}. If $G$ is a non-nilpotent group then there is an abnormal
  maximal subgroup $M$ of $G$ such that $\tilde{F}(G)\nsubseteq M$.\end{cor}

\begin{proof} Assume the contrary. If $\tilde{F}(G)\subseteq M$ for every abnormal
 maximal subgroup $M$ of $G$ then $M$ is  $\tilde{F}(G)$-subnormal.  Thus $G$ is
 nilpotent by theorem 4.1, a contradiction.\end{proof}

The following example shows that we can not use $F^*(G)$ in place of $\tilde{F}(G)$
 in theorem 4.1.  Let $ G \simeq A_{5}$ be the alternating group on
5 letters and $ K = \mathbb{F}_{3}$. According to \cite{20} there is faithful
 irreducible Frattini $KG$-module $A$  of  dimension 4. By known Gaschutz
 theorem \cite{41}, there
exists a Frattini extension  $A\rightarrowtail R\twoheadrightarrow G$
such that $A\stackrel {G}{\simeq} \Phi(R)$ and $R/\Phi(R)\simeq G$.
From the properties of module $A$ it follows that $\tilde {F}^\infty(G) =\tilde {F}(G) = R$
 and  $F^{*}(G)=\Phi(R)$.

\textbf{Problem 2.} Is it possible to replace $\tilde{F}(G)$ by $\tilde{F}^\infty(G)$
 in theorem 4.1.

\begin{thm}  The following statements for a group $G$ are equivalent:

(1) $ G $ is nilpotent;

(2) Every abnormal subgroup of $ G $ is $ F^{*}(G) $-subnormal
 subgroup of $ G $;

(3) All normalizers  of Sylow subgroups of $G$ are  $ F^{*} (G) $-subnormal;

(4) All cyclic primary subgroups of $G$ are  $ F^{*} (G) $-subnormal;

(5) All  Sylow subgroups of $G$ are $ F^{*} (G) $-subnormal.
\end{thm}

  \begin{proof}  (1)$\Rightarrow$(2). Let $G$ be a nilpotent group. It is clear that every
  subgroup of $G$ is $F^*(G)$-subnormal. Hence (1) implies (2).

 (2)$\Rightarrow$(3). Normalizers of  all Sylow subgroups are abnormal. Therefore
(2) implies (3).

 (3)$\Rightarrow$(4). Every cyclic primary subgroup $H$ of $G$ is contained in some Sylow
 subgroup $P$ of $G$. So
 $H\triangleleft\triangleleft P\triangleleft N_G(P)\triangleleft \triangleleft N_G(P)F^*(G)$.
 Hence $H\triangleleft\triangleleft N_G(P)F^*(G)$. In particular, $H$ is subnormal in $HF^*(G)$.
 Hence (3) implies (4).

 (4)$\Rightarrow$(5). Let $ P$ be a Sylow subgroup of $G$ and  $ x \in P $. Then the subgroup
 $ \langle x \rangle $ is the $ F^{*} (G) $-subnormal subgroup. So
 $ \langle x \rangle \triangleleft \triangleleft \langle x \rangle F^{*} (G) $. Note
that $ \langle x \rangle \triangleleft \triangleleft P $. Since
$ \langle x \rangle \leq P \cap \langle x \rangle F^{*} (G) $, by
theorem 1.1.7  (\cite{6} p.3) $ \langle x \rangle $ is the subnormal subgroup in the product
 $ P (\langle x \rangle F^{*} (G))=PF^*(G) $.
 Since $ P $ is generated by its cyclic subnormal in $ P F^{*} (G) $ subgroups, we see
 that $ P \triangleleft \triangleleft PF^{*} (G) $. So (4) implies (5).

(5)$\Rightarrow$(1).
Let $P$ be an arbitrary Sylow subgroup of $G$.  Since $P$ is pronormal subnormal subgroup of
  $PF^*(G)$, by lemma 6.3 p.241 of \cite{HD} we see that $F^*(G)\leq N_G(P)$. Now $F^*(G)$ lies
 in the intersection of all normalizers of Sylow subgroups of $G$. By result of \cite{15},
 $ F^{*} (G) \subseteq Z_{\infty} (G)$.  By lemmas  2.2 and 2.4 we  see that
 $G^\mathfrak{N}\leq Z_\infty(G)$. Thus $G$ is nilpotent by lemma 2.3. \end{proof}

\begin{cor} A group $ G $ is nilpotent if and only if
 the normalizers of all Sylow subgroups of $ G $ contains
$ F^{*}(G) $.\end{cor}

\begin{proof} If $G$ is nilpotent then for every Sylow subgroup $P$ of $G$ we have $F^*(G)=N_G(P)=G$.
 Thus the normalizers of all Sylow subgroups of $ G $ contains
$ F^{*}(G) $. If   the normalizers of all Sylow subgroups of $ G $ contains
$ F^{*}(G) $ then they are $ F^{*}(G) $-subnormal. Thus $G$ is nilpotent. \end{proof}

\begin{thm} If a group $G$ is a product of two $F(G)$-subnormal nilpotent subgroups $A$ and $B$
 then $G$ is
 nilpotent.\end{thm}

\begin{proof} Since $G$ is a product of two nilpotent subgroups, we see that $G$ is soluble. Hence
 $F(G)=F^*(G)$.   Let $T=AF(G)$. Note that $T$ is a product of two subnormal nilpotent subgroups.
 Hence,
 $T$ is nilpotent. By analogy $R=BF(G)$ is nilpotent and $G=TR$. Let $p$ be a prime. If $T_p$ is
 a Sylow
 $p$-subgroup of $T$ then $T_p\triangleleft T$, in particular $F(G)\leq N_G(T_p)$. By analogy if
 $R_p$ is
 a Sylow $p$-subgroup of $R$ then $F(G)\leq N_G(R_p)$. It is well known that $R_pT_p=G_p$ is a Sylow
 $p$-subgroup of $G=RT$. Now $F(G)\leq N_G(G_p)$. From $F(G)\triangleleft G$ it follows that  $F(G)$
 lies in all normalizers of Sylow $p$-subgroups of $G$. So we see that $F(G)$ lies  in all normalizers
 of
 Sylow subgroups of $G$. Thus by corollary 4.4 $G$ is nilpotent.\end{proof}

\begin{thm} Let a group $G=\langle A,B\rangle$ where $A$ and $B$  are $F(G)$-subnormal supersoluble
 subgroups of $G$. If $[A,B]$ is nilpotent, then $G$ is supersoluble.\end{thm}

\begin{proof}   By (1) of lemma 2.7 $[A,B]\triangleleft G$. By (4) of lemma 2.7 $[A,B]\leq F(G')$.
  First we will show that $G'$ is nilpotent. Since $A$ is supersoluble, $A'$ is nilpotent normal
 subgroup of
 $A$. From $A\triangleleft\triangleleft AF(G)$ we obtain that $A'\triangleleft\triangleleft AF(G)$.
 In particular $A'\triangleleft\triangleleft A'F(G)$. Note that $F(G')\leq F(G)$. Thus $A'$ is
 $F(G')$-subnormal nilpotent subgroup of $G'$. Note that $A'F(G')$ is $F(G')$-subnormal nilpotent
   subgroup of $G'$ and $[A,B]\leq A'F(G')$. By analogy $B'F(G')$ is $F(G')$-subnormal nilpotent
  subgroup of $G'$ and $[A,B]\leq B'F(G')$. Thus $G'$ is a product of two nilpotent $F(G')$-subnormal
 subgroups by (4) of lemma 2.7 and hence by theorem 4.5 is nilpotent.

Now we are going to show that $G$ is supersoluble. Note that $[A,B]\leq F(G)$. So $A$ is
 $[A,B]$-subnormal.
 Hence $A[A,B]$ is a product of supersoluble subnormal subgroup $A$ and nilpotent normal
 subgroup $[A,B]$.
 Thus $A[A,B]$ is normal supersoluble subgroup of $G$ by lemmas 2.6 and 2.7. By analogy $B[A,B]$
 is normal
 supersoluble subgroup of $G$. So $G$ is a product of two normal supersoluble subgroups and $G'$
 is nilpotent,
 i.e. $G$ is supersoluble by lemma 2.5. \end{proof}

The well known result states that a group $G$ is supersoluble  if it contains two normal supersoluble
 subgroups
 with coprime indexes in $G$. As the following example shows this result fails if we replace
‘‘normal’’ by ‘‘$F(G)$-subnormal’’.  Let $G$ be the symmetric group on 3 letters. By theorem 10.3B
 there is
 a faithful irreducible $\mathbb{F}_{7}G$-module $V$ and the dimension of $V$ is 2. Let $R$ be the
 semidirect
 product of $ V$ and $G$. Let $A=VG_{3}$ and $B=VG_{2}$ where $G_{p}$ is a Sylow $p$-subgroup of $
G$, $p\in{2,3}$. Since $7\equiv 1 ($mod$~p)$ for $p\in{2,3}$, it is easy to check that subgroups $A$
 and $B$ are supersoluble. Since $V$ is faithful irreducible module, $F(R)=V$. Therefore $A$ and $B$
 are the $F(R)$-conjugate-permutable subgroups of $G$. Note that $R=AB$ but $R$  is not supersoluble.

Let $p$ be a prime. Recall that a group $G$ is called $p$-decomposable if it is the direct product
 of its
 Sylow $p$-subgroup on Hall $p'$-subgroup.

\begin{lem} Let $p$ be a prime. The $p$-decomposable residual of a group $G$ is generated by all
 commutators $[a,b]$, where $a$ is a $p$-element and $b$ is a $p'$-element of prime power
 order.\end{lem}

\begin{proof} Let $N$ be a subgroup of $G$  generated by all commutators $[a,b]$, where $a$ is a
 $p$-element and $b$ is a $p'$-element of prime power order. It is clear that $N\triangleleft G$.
 Let $x$ be a  $p$-element and $y$ be a $p'$-element of prime power order in $G$. Then
 $xNyN=xyN=yx[x,y]N=yxN$.
 Hence in $G/N$ every $p$-element commutes with every $p'$-element of prime power order. But every
 $p'$-element of $G/N$ is a product of $p'$-elements of prime power order. Thus every  $p$-element
 commutes with every $p'$-element in $G/N$. So $G/N$ is $p$-decomposable.

Assume now that $N$ is $p$-decomposable residual of $G$,       $a$ is a $p$-element and $b$ is
 a $p'$-element of prime power order of $G$. Then $aNbN=bNaN$ or $N=[a,b]N$. Hence $[a,b]\in N$.
 This completes the proof.\end{proof}

\begin{thm}Let $A$, $B$ and $C$ be a $F(G)$-subnormal supersoluble subgroups of a group $G$.
 If indexes
  of $A$, $B$ and $C$ in $G$ are pairwise coprime then $G$ is supersoluble.\end{thm}

\begin{proof} Let $|G|=p_1^{a_1}p_2^{a_2}\dots p_n^{a_n}$ where $p_i$ is a prime
 and $p_1>p_2>\dots >p_n$.
 Denote through $G_{p}$ Sylow $p$-subgroup of $G$ where $p\in\{p_1,\dots,p_n\}$.

Since $A$, $B$ and $C$ are supersoluble, they satisfy the Sylow tower property (see \cite{22}, p.5).
 So
  $G$ contains three subgroups with coprime indexes that satisfies the Sylow tower property.
 By theorem
 4.13 (\cite{13} p.47) $G$ satisfies the Sylow tower property. By lemma 2.6 subgroups
 $A_1=AF(G)$, $B_1=BF(G)$
 and $C_1=CF(G)$   are supersoluble. Note that $G_{p_1}$ is a normal subgroup of $G$.
 So it is contained
 in $A_1$, $B_1$ and $C_1$.

 Assume that  supersoluble subgroups $A_i$, $B_i$, $C_i$ with coprime indexes in $G$ contain
 $G_{p_1}G_{p_2}\dots G_{p_i}$ and $F(G)$. Let us show that there are supersoluble subgroups
 $A_{i+1}$, $B_{i+1}$, $C_{i+1}$ with coprime indexes in $G$ that contain
 $G_{p_1}G_{p_2}\dots G_{p_{i+1}}$
 and $F(G)$.

By coprime indexes hypothesis we see that at least  two subgroups of    $A_i$, $B_i$, $C_i$
 contain
 $G_{p_1}G_{p_2}\dots G_{p_{i+1}}$. Without lose of generality let
 $G_{p_1}G_{p_2}\dots G_{p_{i+1}}\leq A_i$
 and $G_{p_1}G_{p_2}\dots G_{p_{i+1}}\leq B_i$.

 Note that if $h$ is a prime power order element of $G$ then $h$ is contained
 in $A_i^x$ or $B_i^x$ for
 some $x\in G$. Since $G_{p_1}G_{p_2}\dots G_{p_{i+1}}\triangleleft G$ all
 $p_{i+1}$-elements of $G$ are
 contained in $A_i^x$ and $B_i^x$ for all $x\in G$. It means that $[a,b]$ where $a$
 is a $p_{i+1}$-element
 and $b$ is a $p_{i+1}'$-element of prime power order belong to the $p_{i+1}$-decomposable
 residual of $A_i^x$
 or $B_i^x$ for some $x\in G$ by lemma 4.7. Therefore $p_{i+1}$-decomposable residual of $G$
 is generated by
 all  $p_{i+1}$-decomposable residuals of  $A_i^x$ and $B_i^x$ for all $x\in G$ by lemma 4.7.
 Note that $A_i^x$
 or $B_i^x$ are supersoluble and, hence, metanilpotent for all $x\in G$. It is well known that
 $p$-decomposable
 residual is a part of nilpotent residual. Hence  $p_{i+1}$-decomposable residuals of
  $A_i^x$ and $B_i^x$
 are nilpotent for all $x\in G$. Note that $[a,b]=a^{-1}a^b$. So $p_{i+1}$-decomposable
 residuals of  $A_i^x$
 and $B_i^x$ lie in $G_{p_1}G_{p_2}\dots G_{p_{i+1}}\triangleleft G$ for all $x\in G$.
 Thus they all
 are nilpotent subnormal subgroups of $G$. So $p_{i+1}$-decomposable residual of $G$
 lies in $F(G)$.

Let $T=G_{p_{i+1}}F(G)$. Since $p_{i+1}$-decomposable residual of $G$ lies in $F(G)$,
 we see that
 $T\triangleleft G$. Note that  $T\leq A_i$ and hence supersoluble. Let $C_{i+1}=C_iT$.
 We see that indexes
 of $C_{i+1}$, $A_{i}$ and $B_i$ are pairwise coprime in $G$. Since $C_i$ is supersoluble,
 there is Hall
 $p_{i+1}'$-subgroup $H$ of $C_i$. Let $R=HF(G)$. From $R\leq C_i$ we see that $R$ is supersoluble.
 Note
 that indexes of $T$ and $R$ are coprime in $C_{i+1}$.   So $C_{i+1}=RT$. Note that $C_{i+1}/F(G)$
 is
 $p_{i+1}$-decomposable, $T/F(G)$ is a Sylow $p_{i+1}$-subgroup of $C_{i+1}/F(G)$ and $R/F(G)$
 is a Hall
 $p_{i+1}'$-subgroup of $C_{i+1}/F(G)$. It is clear that $T/F(G)\triangleleft C_{i+1}/F(G)$ and
 $R/F(G)\triangleleft C_{i+1}/F(G)$. Thus  $T\triangleleft C_{i+1}$ and $R\triangleleft C_{i+1}$.
 So
 $C_{i+1}$ contains two normal supersoluble subgroups with coprime indexes and hence $C_{i+1}$
 is itself
 supersoluble by theorem 3.4  (\cite{22} p.127).

Now supersoluble subgroups $A_{i+1}=A_i$, $B_{i+1}=B_i$, $C_{i+1}$ have coprime indexes in $G$,
 contain
 $G_{p_1}G_{p_2}\dots G_{p_{i+1}}$ and $F(G)$. Since $G=G_{p_1}G_{p_2}\dots G_{p_{n}}$,
 we see that $G$ is
 itself supersoluble.\end{proof}

\begin{cor}Let $A$, $B$ and $C$ be a  supersoluble subgroups of a group $G$. If indexes
  of $A$, $B$ and
 $C$ in $G$ are pairwise coprime and $F(G)\leq A\cap B\cap C$ then $G$ is supersoluble.\end{cor}

\section{Final remarks}

Let $\mathfrak{F}$ be an $N_0$-closed class of groups and $1\in\mathfrak{F}$. Then there is
 the maximal
 normal $\mathfrak{F}$-subgroup $G_\mathfrak{F}$ in any group $G$. In the context of our work
 the following
 general problem appears: study all $N_0$-closed classes (formations, Fitting classes, Shunck classes)
 $\mathfrak{F}$ with 1 for which one of the following statements holds:

(1) $F(G)\subseteq G_\mathfrak{F}\subseteq F^*(G)$ for any group $G$;

(2) $F^*(G)\subseteq G_\mathfrak{F}\subseteq \tilde{F}(G)$ for any group $G$;

(3) $F(G)\subseteq G_\mathfrak{F}\subseteq \tilde{F}(G)$ for any group $G$.

 The motivation of this problem is the following theorem.

\begin{thm} Let $\mathfrak{F}$ be a  $N_0$-closed formation

(1) If $\mathfrak{F}$ is saturated and $F(G)\subseteq G_\mathfrak{F}\subseteq \tilde{F}(G)$ for any
 group
 $G$ then $\mathfrak{F}=\mathfrak{N}$.

(2) If $F^*(G)\subseteq G_\mathfrak{F}\subseteq \tilde{F}(G)$ for any group $G$ then
 $\mathfrak{F}=\mathfrak{N}^*$.

\end{thm}

\begin{proof}

Let prove (1).
 From $F(G)\subseteq G_\mathfrak{F}$ for every group $G$ it follows that
 $\mathfrak{N}\subseteq \mathfrak{F}$.
  Assume that $\mathfrak{F}\setminus\mathfrak{N}\neq\emptyset$. Let us choose a minimal order
 group $G$
 from $\mathfrak{F}\setminus\mathfrak{N}$. Since $\mathfrak{F}$ and $\mathfrak{N}$ are both
 saturated
 formations, from the choice of $G$ we may assume that $\Phi(G)=1$ and there is only one
 minimal normal
 subgroup of $G$. From $G_\mathfrak{F}\subseteq \tilde{F}(G)$ it follows that $G=Soc(G)$ is
 nonabelian simple
 group. From \cite{20} it follows that for a prime $p$ dividing $|G|$ there there exist a faithful
 $\mathbb{F}_pG$-module $A$ admitting a group extension $A\rightarrow E\twoheadrightarrow G$ with
 $A\stackrel {G}{\simeq} \Phi(E)$ and $E/\Phi(E)\simeq G$. So $E\in\mathfrak{F}$ and $\Phi(E)/1$
 is the normal section  of $E$. According to \cite{BK} we see that
  $H=\Phi(E)\leftthreetimes (E/\Phi(E))\in\mathfrak{F}$. Note
 that $\tilde{F}(H)\simeq\Phi(E)<H=H_\mathfrak{F}$, a contradiction.
 Thus $\mathfrak{N}=\mathfrak{F}$.

Let prove (2).
 From $F^*(G)\subseteq G_\mathfrak{F}$ for every group $G$ it follows that
 $\mathfrak{N}^*\subseteq \mathfrak{F}$.  Assume that
 $\mathfrak{F}\setminus\mathfrak{N}^*\neq\emptyset$.
 Let us choose a minimal order group $G$ from $\mathfrak{F}\setminus\mathfrak{N}^*$.
 It is clear
 that $G=G_\mathfrak{F}=\tilde{F}(G)$.  Since $\mathfrak{F}$ and $\mathfrak{N}^*$
 are both  formations,
 from the choice of $G$  there is only one minimal normal subgroup $N$  in $G$.
 If $\Phi(G)=1$
 then $G=Soc(G)\in\mathfrak{N}^*$, a contradiction. So $N\leq\Phi(G)$.
Now $N$ is normal elementary abelian $p$-subgroup $G$.  By our assumption $G/N\in\mathfrak{N}^*$.
 Assume
 that $C_G(N)=G$. Now $G$ acts as inner automorphisms on $N/1$ and on every chief factor of $G/N$.
 By definition of quasinilpotent groups $G\in\mathfrak{N}^*$, a contradiction. Hence $C_G(N)\neq G$.
 Note that $N$ is the unique minimal subgroup of  $H=N\leftthreetimes (G/C_G(N))\in\mathfrak{F}$ by
 \cite{BK} and $\Phi(H)=1$. So $\tilde{F}(G)=N$ and $H_\mathfrak{F}=H$, a contradiction.
 Thus $\mathfrak{N}^*=\mathfrak{F}$.
\end{proof}

\textbf{Problem 3.} Describe all Fitting classes $\mathfrak{F}$ for which one of the following holds.

(1) $F(G)\subseteq G_\mathfrak{F}\subseteq F^*(G)$ for any group $G$;

(2) $F^*(G)\subseteq G_\mathfrak{F}\subseteq \tilde{F}(G)$ for any group $G$;

(3) $F(G)\subseteq G_\mathfrak{F}\subseteq \tilde{F}(G)$ for any group $G$.

 In \cite{F1} Forster introduced a class
 $\widehat{\mathfrak{N}^*}=E_\Phi\mathfrak{N}^*=(G|\tilde{F}(G)=G)$
 and showed that $\widehat{\mathfrak{N}^*}$ is $N_0$-closed Shunck class that
 is neither formation nor
 $s_n$-closed. Note that $\widehat{\mathfrak{N}^*}=(G|\tilde{F}^\infty(G)=G)$.

 \begin{prop} Let $G$ be a group. Then $G_{\widehat{\mathfrak{N}^*}}=\tilde{F}^\infty(G)$,
 i.e.
 $\tilde{F}^\infty(G)$ is the maximal among normal subgroups $N$  of $G$ such
 that $N/\Phi(N)$ is
 quasinilpotent.\end{prop}

\begin{proof} From $G_{\widehat{\mathfrak{N}^*}}/\Phi(G_{\widehat{\mathfrak{N}^*}})\in\mathfrak{N}^*$
 and $\Phi(G_{\widehat{\mathfrak{N}^*}})\leq\Phi(G)$ it follows that
 $G_{\widehat{\mathfrak{N}^*}}\leq\tilde{F}(G)$. By induction
 $G_{\widehat{\mathfrak{N}^*}}=\tilde{F}(G_{\widehat{\mathfrak{N}^*}})\leq \tilde{F}^\infty(G)$.
 By proposition 3.2 and the definition of $\widehat{\mathfrak{N}^*}$ we obtain
 $G_{\widehat{\mathfrak{N}^*}}=\tilde{F}^\infty(G)$. \end{proof}

Now $F^*(G)\subseteq G_{\widehat{\mathfrak{N}^*}}\subseteq \tilde{F}(G)$. So we can ask

\textbf{Problem 4.} Describe all $N_0$-closed Shunck classes $\mathfrak{F}$ for which one of
 the following holds.

(1) $F(G)\subseteq G_\mathfrak{F}\subseteq F^*(G)$ for any group $G$;

(2) $F^*(G)\subseteq G_\mathfrak{F}\subseteq \tilde{F}(G)$ for any group $G$;

(3) $F(G)\subseteq G_\mathfrak{F}\subseteq \tilde{F}(G)$ for any group $G$.

\end{document}